\documentclass[11pt,a4paper]{article}

\usepackage{inputenc}
\usepackage{amsmath}
\usepackage{bm}
\usepackage{bbold}
\usepackage{amsthm}
\usepackage{enumerate}

\usepackage[pdfstartview={XYZ null null 1.00},pdfpagelayout=SinglePage,pdfpagemode=UseNone,breaklinks]{hyperref}

\title{Minimizing maximum lateness in\\ two-stage projects by tropical optimization\thanks{Kybernetika 58(5), 816-841 (2022). https://www.kybernetika.cz/content/2022/5/816}}

\author{N.~Krivulin\thanks{Saint Petersburg State University, Faculty of Mathematics and Mechanics, 28 Universitetsky Ave., St.~Petersburg 198504, Russia, nkk@math.spbu.ru.}
\and
S.~Sergeev\thanks{University of Birmingham, School of Mathematics, Edgbaston B15 2TT, UK, s.sergeev@bham.ac.uk.}
}

\date{}

\newtheorem{theorem}{Theorem}
\newtheorem{lemma}[theorem]{lemma}
\newtheorem{corollary}[theorem]{corollary}
\newtheorem{proposition}[theorem]{Proposition}

\theoremstyle{definition}

\begin{document}

\maketitle

\begin{abstract}
We are considering a two-stage optimal scheduling problem, which involves two similar projects with the same starting times for workers and the same deadlines for tasks. It is required that the
starting times for workers and deadlines for tasks should be optimal for the first-stage project and, under this condition, also for the second-stage project. Optimality is measured with respect to the maximal lateness (or maximal delay) of tasks, which has to be minimized. We represent this problem as a problem of tropical pseudoquadratic optimization and show how the existing methods of tropical optimization and tropical linear algebra yield a full and explicit solution for this problem.
\\

\\

\end{abstract}

\section{Introduction}


As an important component of project management, project scheduling is concerned with the development of optimal schedules of activities that comprise a project, subject to various constraints \cite{Demeulemeester2002Project,Neumann2003Project,Tkindt2006Multicriteria}. The scheduling objectives are normally formulated in terms of time-oriented criteria to optimize such as makespan, lateness and tardiness. In multicriteria and multilevel scheduling, further objectives can be added, which take into account the project cost, profit, resource allocation or consumption, etc. The scheduling constraints may include temporal constraints (time bounds for and relationships between activities) as well as material and manpower resource requirements, budget limitations and others.

In the general case, the project scheduling problems with constraints of different types may be rather complicated and even known to be $NP$-hard to solve. The solution approaches used to handle these problems involve methods and computational schemes of linear, integer and mixed integer linear programming, combinatorial and discrete optimization, which offer algorithmic technique to obtain exact or approximate numerical solutions \cite{Demeulemeester2002Project,Tkindt2006Multicriteria}.

The project scheduling problems, which have only time-oriented objectives and temporal constraints, can normally be formulated and solved as linear programs using computational algorithms of linear programming. This approach typically offers a quite efficient numerical technique to find a solution or determine that no solution exists, but does not allow to derive all solutions analytically.

Another approach to solve temporal scheduling problems is based on models and methods of tropical (idempotent) algebra which is concerned with the theory and applications of semirings with idempotent addition \cite{Baccelli1993Synchronization,Butkovic2010Maxlinear,Golan2003Semirings,Gondran2008Graphs,Heidergott2006Maxplus,Kolokoltsov1997Idempotent}. Project scheduling problems served to motivate early research \cite{Cuninghamegreen1962Describing,Giffler1963Scheduling} and still present an important application domain for tropical mathematics \cite{Bouquard2006Application,Butkovic2010Maxlinear,Goto2009Robust}.

In the framework of tropical algebra, many temporal project scheduling problems can be represented as optimization problems and then analytically solved using methods and techniques of tropical optimization in explicit form. Examples of the solution include results on both single-criterion problems \cite{Krivulin2015Extremal,Krivulin2017Direct,Krivulin2017Tropical,Krivulin2017Tropicaloptimization} and multicriteria problems \cite{Krivulin2020Tropical}.

A more theoretical motivation for this work comes from a general idea to develop the theory and practice of multi-criteria and multi-level tropical optimization problems.  In view of the above mentioned development of tropical optimization problems applied to project scheduling, as well as successful development of algorithms to solve some other types of tropical optimization problems such as tropical linear programming and tropical linear-fractional programming \cite{TropicalSimplex, TropicalMPG}, it seems to be a promising research direction. Note, however, that except for the tropical bi-objective problem of \cite{Krivulin2020Tropical} and the tropical bi-level optimization problem of \cite{SLiu-20}, we are not aware of any other tropical bi-objective or bi-level optimization problems considered in the existing literature.

In this paper, a two-stage project scheduling problem under temporal objectives and constraints is considered. This problem can be considered as a lexicographic bi-objective optimization problem \cite{Ehrgott2005Multicriteria}, where the objectives are taken into account in a hierarchical order. We show how this problem can be represented in terms of tropical algebra as a multidimensional two-stage tropical optimization problem. We completely solve this problem, that is, we derive an analytical solution in compact vector form and evaluate the computational complexity of the solution.

The rest of our paper is organized as follows. In Section~\ref{s:scheduling}, we give some background on project scheduling problems and formulate the problem, which we are going to solve in the present paper. In Section~\ref{s:algebra}, we recall some basic operations of tropical linear algebra, in particular, tropical traces and Kleene stars. While most of the material given there is well-known (see the references in the text), the method that we use in computation of tropical traces in Subsection~\ref{ss:traces} was developed only very recently and here we give a more detailed analysis of its computational complexity. The solution of the scheduling problem is given in Section~\ref{s:solution}, see Theorem~\ref{T:upper-level} for the main result. In Section~\ref{s:complexity}, we discuss the computational complexity of the solution.

\section{Project scheduling problems}
\label{s:scheduling}

In this section, we describe one- and two-stage project scheduling problems that motivate the research and illustrate the application of the results. As a real-world example that provides an appropriate applied context for the study, one can consider a problem of optimal scheduling health care activities in a hospital, where each patient is assigned to one or more courses of treatment (medications, therapeutic procedures, diagnostic tests, etc.), and each course is carried out individually on one or together on more patients. The patient is discharged from the hospital as soon as all prescribed treatments are completed.

Each course involves a basic component (primary treatment) and auxiliary component (secondary treatment), which run simultaneously. In accordance to medical, technological, regulatory, managerial or other requirements, a set of temporal constraints is imposed on the start time of the courses and the discharge date of the patients. For each pair of a patient and a course, the prescribed duration of the course of the primary and secondary treatments is specified. Furthermore, the maximum time lag between the start time of the components of the course and the discharge due date of the patient is also given. Finally, lower and upper bounds (box constraints) are imposed on the start date for both components of each course and on the due date for each patient.

The primary treatment components of the courses are scheduled first to obtain optimal sets of the start dates for courses and the corresponding discharge due dates for patients. The schedule has to minimize the maximum lateness of the course completion dates with respect to due dates over all patients, subject to the temporal constraints described above. Next, the set of optimal solutions obtained for this problem are used in forming the feasible set to solve the scheduling problem for the secondary treatment components under the same minimum lateness criterion.

In the rest of the section, we show how this scheduling problem can be represented as a minimax two-stage optimization problem. To provide a general description of the scheduling model, we use the terms tasks and workers instead of patients and treatments appeared in the above application example.

\subsection{One-stage project scheduling problems}

Consider a project (a treatment plan in the hospital setting) that involves $m$ tasks (patients) to be performed (treated) in parallel by $n$ workers (courses of treatment), subject to temporal constraints in the form of start-finish and due date-start precedence relationships, time boundaries on the start times of workers and the due dates of tasks. The start-finish constraints specify the minimum allowed time lag for each pair worker-task between the start of the worker (the beginning of the course) and the finish of the task (the discharge of the patient). Each task is assumed to finish immediately after all start-finish constraints for its finish time are satisfied.

The due date-start constraints specify the minimum allowed time lag for each pair task-worker between the due date of the task (the due date of the patient's discharge) and the start of the worker (the beginning of the course). The boundary constraints include the release times and release deadlines for the workers to start (for the courses to begin), and the earliest allowed time and deadlines for the tasks to finish (for the patients to discharge). The above constraints are considered as hard restrictions, which cannot be violated.

In contrast to the above constraints, the due time constraints indicate the desirable times for tasks to finish, and are not set in advance. The project scheduling problems of interest are to determine the start time for each worker and the due time for each task to meet the minimum lateness objective, which is to minimize the maximum difference between the finish time and the due time over all tasks.

In the hospital context, the objective aims at the development of a schedule of the treatment plan for all patients to minimize the maximum lateness of courses with respect to discharge due dates of patients.

For each task $i=1,\ldots,m$, we denote its finish time by $f_{i}$, and for each worker $j=1,\ldots,n$, we denote his or her start time by $x_{j}$. Let $a_{ij}$ be the minimum allowed time lag between the start of worker $j$ and the finish of task $i$. Then, the start-finish constraints for task $i$ take the form of inequalities $a_{ij}+x_{j}\leq f_{i}$ for all $j=1,\ldots,n$, one of which is an equality. Combining all inequalities for each $j$ yields the equalities
\begin{equation*}
\max_{1\leq j\leq n}(a_{ij}+x_{j})
=
f_{i},
\quad
i=1,\ldots,m.
\end{equation*}

For each task $i=1,\ldots,m$, we denote its due time by $y_{i}$. Furthermore, let $-b_{ji}$ denote the minimum (negative) allowed time lag between the due time of task $i$ and the start time of worker $j$. The due date-start constraints are given by the inequalities $-b_{ji}+y_{i}\leq x_{j}$ for all $i=1,\ldots,m$. Equivalently, they can be written as $y_i\leq b_{ji}+x_j$: the constraints that do not allow the due dates to be too lax with respect to workers' starting times. These constraints can be combined into the inequalities
\begin{equation*}
\max_{1\leq i\leq m}(-b_{ji}+y_{i})
\leq
x_{j},
\quad
j=1,\ldots,n.
\end{equation*}

Let $q_{i}$ and $r_{i}$ be the earliest allowed time to finish and deadline for task $i=1,\ldots,m$, and $g_{j}$ and $h_{j}$ be the release time and release deadline for worker $j=1,\ldots,n$. The due time of tasks and start time of workers must satisfy the box constraints
\begin{equation*}
q_{i}\leq y_{i}\leq r_{i},
\quad
i=1,\ldots,m;
\qquad
g_{j}\leq x_{j}\leq h_{j},
\quad
j=1,\ldots,n.
\end{equation*}

Finally, the maximum lateness over all tasks is defined as follows:
\begin{equation*}
\max_{1\leq i\leq m}(f_{i}-y_{i})
=
\max_{1\leq i\leq m}\left(\max_{1\leq j\leq n}(a_{ij}+x_{j})-y_{i}\right).
\end{equation*}

Given the parameters $a_{ij}$, $b_{ij}$, $q_{i}$, $r_{i}$, $g_{j}$ and $h_{j}$ for all $i=1,\ldots,m$, and $j=1,\ldots,n$, the project scheduling problem can be formulated to find the unknown vectors of the start time of workers $\bm{x}=(x_{1},\ldots,x_{n})^{T}$, and of the due time for tasks $\bm{y}=(y_{1},\ldots,y_{m})^{T}$, for which the maximum lateness is minimal:
\begin{equation}
\begin{aligned}
\min_{\bm{x},\bm{y}}
&&&
\max_{1\leq i\leq m}
\left(\max_{1\leq j\leq n}(a_{ij}+x_{j})-y_{i}\right);
\\
\text{s.t.}
&&&
\max_{1\leq i\leq m}(-b_{ji}+y_{i})
\leq
x_{j},
\quad
g_{j}\leq x_{j}\leq h_{j},
\quad
j=1,\ldots,n;
\\
&&&
q_{i}\leq y_{i}\leq r_{i},
\quad
i=1,\ldots,m.
\end{aligned}
\label{P-Upper-Level}
\end{equation}

We now assume that there is another project with the same numbers $m$ and $n$ of tasks and workers, and with the same box constraints for start times of workers and due dates for tasks. For each task $i=1,\ldots,m$ in this project, we denote its due date by $v_{i}$, and for each worker $j=1,\ldots,n$, we denote the start time by $u_{j}$. Let $c_{ij}$ be the minimum time lag between the start time of worker $j$ and the finish time of task $i$, and let $-d_{ij}$ be the minimum (negative) time lag between the due date for task $j$ and the start time of worker $i$.

Suppose that the parameters $c_{ij}$, $d_{ij}$, $q_{i}$, $r_{i}$, $g_{j}$ and $h_{j}$ are given for all $i=1,\ldots,m$, and $j=1,\ldots,n$, and consider the problem of finding vectors of start times $\bm{u}=(u_{1},\ldots,u_{n})^{T}$ and of due dates $\bm{v}=(v_{1},\ldots,v_{m})^{T}$ that minimize the maximum lateness in the project:
\begin{equation}
\begin{aligned}
\min_{\bm{u},\bm{v}}
&&&
\max_{1\leq i\leq m}
\left(\max_{1\leq j\leq n}(c_{ij}+u_{j})-v_{i}\right);
\\
\text{s.t.}
&&&
\max_{1\leq i\leq m}(-d_{ji}+v_{i})
\leq
u_{j},
\quad
g_{j}\leq u_{j}\leq h_{j},
\quad
j=1,\ldots,n;
\\
&&&
q_{i}\leq v_{i}\leq r_{i},
\quad
i=1,\ldots,m.
\end{aligned}
\label{P-Lower-Level}
\end{equation}

Note that problems \eqref{P-Upper-Level} and \eqref{P-Lower-Level} can be readily rewritten as linear programs, and solved by appropriate computational techniques of linear programming. However, in the general case, this approach implements numerical iterative procedures and does not offer direct exact solutions in an explicit form, which we are going to use.

\subsection{Two-stage project scheduling problem}

Suppose that both projects described above have common start times for workers and due dates for tasks. Moreover, assume one of the projects to be a first-stage project, for which the start times and due dates must be optimal to minimize the maximum lateness. The other project is a second-stage project such that its start times and due dates must be optimal for the first-stage project and, only under this condition, for the second-stage project itself.

We then represent the overall problem as a two-stage optimization problem with the second-stage problem formulated as \eqref{P-Upper-Level}, and the first-stage problem as \eqref{P-Lower-Level}:
\begin{equation}
\begin{aligned}
\min_{\bm{x},\bm{y}}
&&&
\max_{1\leq i\leq m}
\left(\max_{1\leq j\leq n}(a_{ij}+x_{j})-y_{i}\right);
\\
\text{s.t.}
&&&
\max_{1\leq i\leq m}(-b_{ji}+y_{i})
\leq
x_{j},
\quad
j=1,\ldots,n;
\\
&&&
(\bm{x},\bm{y})
\in
\arg\min_{\bm{u},\bm{v}}
\bigg\{
\max_{1\leq i\leq m}\left(\max_{1\leq j\leq n}(c_{ij}+u_{j})-v_{i}\right):
\\
&&&\qquad\qquad
\max_{1\leq i\leq m}(-d_{ji}+v_{i})
\leq
u_{j},
\quad
g_{j}\leq u_{j}\leq h_{j},
\quad
j=1,\ldots,n;
\\
&&&\qquad\qquad
q_{i}\leq v_{i}\leq r_{i},
\quad
i=1,\ldots,m
\bigg\}.
\end{aligned}
\label{P-Bi-Level-Simple}
\end{equation}

In the hospital setting, the two-stage problem is to minimize the maximum lateness of courses with respect to the discharge due dates of patients in a two-stage treatment plan, subject to a set of given temporal constraints.

This is the main problem which we are going to solve in this paper. In the next sections, we will reformulate it in terms of tropical algebra. Its explicit analytical solution will be then presented in Section \ref{s:solution}.

\section{Tropical algebra}
\label{s:algebra}

\subsection{Idempotent semirings}

Max-plus semiring $\mathbb{R}_{\max,+}$ is the set $\mathbb{R}\cup\{-\infty\}$ equipped with the associative and commutative operations of addition $a\oplus b=\max(a,b)$ and multiplication $a\otimes b=a+b$, where $\mathbb{1}=0$ is the unit element (i.\,e., such that $\mathbb{1}\otimes a=a$ for all $a\in\mathbb{R}$) and $\mathbb{0}=-\infty$ is the zero element (i.\,e., neutral with respect to addition and such that $\mathbb{0}\otimes a=\mathbb{0}$ for all $a\in\mathbb{R}_{\max,+}$). Observe that the addition is idempotent: $a\oplus a=a$, and the multiplication is distributive over addition and invertible: $a^{-1}=-a$ for any $a\in\mathbb{R}$.

Max-plus semiring is an example of idempotent semifield $(\mathbb{X},\oplus,\otimes,\mathbb{0},\mathbb{1})$ which is defined as a set $\mathbb{X}$ closed under the associative and commutative operations: idempotent addition $\oplus$ and invertible multiplication $\otimes$ (the multiplication sign $\otimes$ is normally omitted in written expressions). This semifield is assumed algebraically closed meaning that, for any $a\in\mathbb{X}$ and natural $m$, the equation $x^{m}=a$ has a solution, where $x^{m}$ denotes the $m$-fold tropical product of $x$ with itself. It is also assumed linearly ordered meaning that the canonical order $a\leq b$ given by the equality $a\oplus b=b$ is total: either the condition $a\leq b$ or the condition $b\leq a$ holds for any $a,b\in\mathbb{X}$.

Both addition and multiplication are monotone with respect to each argument (i.\,e., the inequality $a\leq b$, where $a,b\in\mathbb{X}$ results in the inequalities $a\oplus c\leq b\oplus c$ and $ac\leq bc$ for any $c\in\mathbb{X}$), whereas inversion is antitone (i.\,e., the inequality $a\leq b$, where $a,b\ne\mathbb{0}$, yields $a^{-1}\geq b^{-1}$). Furthermore, addition possesses the extremal property (the majority law) meaning that $a\leq a\oplus b$ and $b\leq a\oplus b$. Finally, the inequality $a\oplus b\leq c$ is equivalent to the pair of inequalities $a\leq c$ and $b\leq c$.

For the sake of our practical application, we only need $\mathbb{X}=\mathbb{R}_{\max,+}$. However, all our theoretical results are also true over any algebraically closed linearly ordered idempotent semifield $\mathbb{X}$ and will be formulated in terms of such semifield.

\subsection{Algebra of matrices}

Scalar operations can be extended to matrices and vectors in the usual way. In particular, let $\bm{A}=(a_{ij})$ and $\bm{B}=(b_{ij})$ be two matrices and $x$ be a scalar over $\mathbb{X}$. Provided that they have appropriate dimensions, one defines
\begin{equation*}
(\bm{A}\oplus \bm{B})_{ij}
=
a_{ij}\oplus b_{ij},
\qquad
(\bm{A}\otimes \bm{B})_{ij}
=
\bigoplus_{k}a_{ik}\otimes b_{kj},
\qquad
(x\otimes\bm{A})_{ij}
=
x\otimes a_{ij}.
\end{equation*}

The matrix multiplication operation gives rise to matrix powers for square matrices over $\mathbb{X}$, defined as
\begin{equation*}
\bm{A}^{k}
=
\underbrace{\bm{A}\otimes\dots\otimes\bm{A}}_{k}.
\end{equation*}

The product sign $\otimes$ will be systematically omitted below.

A matrix $\bm{A}$ of arbitrary dimension (in particular, a vector) is the zero matrix denoted by $\bm{0}$ if it has all entries equal to $\mathbb{0}$. A vector $\bm{x}$ is called regular if it has no zero entries. A matrix is called column-regular (row-regular) if all its columns (rows) are regular, which means that the matrix has no column (row) that consists entirely of zeros.

For a nonzero matrix $\bm{A}=(a_{ij})$, we define the multiplicative inverse transpose (the conjugate) $\bm{A}^{-}=(a_{ij}^{-})$ as the matrix for which $a_{ij}^{-}=a_{ji}^{-1}$ if $a_{ji}\ne\mathbb{0}$, and $a_{ij}^{-}=\mathbb{0}$ otherwise. When applied to nonzero column vectors $\bm{x}=(x_{i})\in\mathbb{X}^{n}$, this operation results in the row vector $\bm{x}^{-}=(x_{i}^{-})$ with the entries $x_{i}^{-}=x_{i}^{-1}$ if $x_{i}\ne\mathbb{0}$, and $x_{i}^{-}=\mathbb{0}$ otherwise.

For any square matrix $\bm{A}=(a_{ij})\in\mathbb{X}^{n\times n}$, the trace is given by
\begin{equation*}
\mathop\mathrm{tr}\bm{A}
=
\bigoplus_{i=1}^{n}a_{ii}.
\end{equation*}

For any matrices $\bm{A}$, $\bm{B}$ and $\bm{C}$ of appropriate size, and a scalar $x$, the following identities hold:
\begin{equation*}
\mathop\mathrm{tr}(\bm{A}\oplus\bm{B})
=
\mathop\mathrm{tr}\bm{A}
\oplus
\mathop\mathrm{tr}\bm{B},
\qquad
\mathop\mathrm{tr}(\bm{A}\bm{C})
=
\mathop\mathrm{tr}(\bm{C}\bm{A}),
\qquad
\mathop\mathrm{tr}(x\bm{A})
=
x\mathop\mathrm{tr}\bm{A}.
\end{equation*}

In order to formulate the solutions of vector inequalities in what follows, we need to introduce the following notion. For any matrix $\bm{A}\in\mathbb{X}^{n\times n}$, we define the following tropical trace function:
\begin{equation*}
\mathop\mathrm{Tr}(\bm{A})
=
\bigoplus_{k=1}^{n}
\mathop\mathrm{tr}\bm{A}^{k},
\label{e:TrA}
\end{equation*}
and then the asterisk operator or the Kleene star of $\bm{A}$ as the series
\begin{equation}
\bm{A}^{\ast}
=
\bigoplus_{k=0}^{\infty}\bm{A}^{k}.
\label{e:kls1}
\end{equation}

The following fact is well-known (see, e.\,g., \cite{Carre1971Algebra,Cuninghamegreen1979Minimax,Yoeli1961Note}).
\begin{proposition}
\label{p:klstrace}
Let $\bm{A}\in\mathbb{X}^{n\times n}$. Series \eqref{e:kls1} converges if and only if $\mathop\mathrm{Tr}(\bm{A})\leq\mathbb{1}$, and in this case
\begin{equation*}
\bm{A}^{\ast}=\bigoplus_{k=0}^{n-1}\bm{A}^{k}.
\end{equation*}
\end{proposition}

Note that the condition $\mathop\mathrm{Tr}(\bm{A})\leq\mathbb{1}$ is equivalent to the condition that the inequality $\mathop\mathrm{tr}\bm{A}^{k}\leq\mathbb{1}$ is valid for each $k=1,\ldots,n$.

\begin{corollary}
\label{c:klstrace}
If $\mathop\mathrm{Tr}(\bm{A})\leq\mathbb{1}$, then for any $r\geq n$, the following equalities hold
\begin{equation*}
\bm{A}^{\ast}
=
\bigoplus_{k=0}^{r-1}\bm{A}^{k},
\qquad
\mathop\mathrm{Tr}(\bm{A})
=
\bigoplus_{k=1}^{r}\mathop\mathrm{tr}\bm{A}^{k}.
\end{equation*}
\end{corollary}

To verify the statement, assume that the condition $\mathop\mathrm{Tr}(\bm{A})\leq\mathbb{1}$ is valid. It follows from Proposition~\ref{p:klstrace} that the inequality
\begin{equation*}
\bm{A}^{r}
\leq\bm{A}^{\ast}
\end{equation*}
holds for all integer $r\geq0$, which yields the first equality.

The second equality is obtained as follows:
\begin{equation*}
\bigoplus_{k=1}^{r}\mathop\mathrm{tr}\bm{A}^{k}
=
\mathop\mathrm{tr}\left(\bm{A}\bigoplus_{k=0}^{r-1}\bm{A}^{k}\right)
=
\mathop\mathrm{tr}(\bm{A}\bm{A}^{\ast})
=
\mathop\mathrm{Tr}(\bm{A}).
\end{equation*}

The notion of the trace function is closely related to the concept of the max-plus algebra spectrum. The max-plus algebra spectral radius of a square matrix $\bm{A}\in\mathbb{X}^{n\times n}$ is given by
\begin{equation*}
\rho(\bm{A})
=
\bigoplus_{k=1}^{n}
\mathop\mathrm{tr}\nolimits^{1/k}(\bm{A}^{k}).
\end{equation*}

This is the largest tropical eigenvalue of $\bm{A}$ (assuming that the order relation $a\leq b$ is given by $a\oplus b=b$), i.\,e., the largest $\lambda$ such that there exists a nonzero vector $\bm{x}\in\mathbb{X}^{n}$ to satisfy the equality $\bm{A}\bm{x}=\lambda\bm{x}$.

It is easy to see that $\mathop\mathrm{Tr}(\bm{A})\leq\mathbb{1}$ if and only if $\rho(\bm{A})\leq\mathbb{1}$.

For any matrices $\bm{A}$ and $\bm{B}$ of appropriate size, the following equality holds:
\begin{equation*}
\rho(\bm{A}\bm{B})
=
\rho(\bm{B}\bm{A}).
\end{equation*}

To verify this equality, it is sufficient to note that each eigenvalue of the matrix $\bm{A}\bm{B}$ is an eigenvalue of $\bm{B}\bm{A}$ and vice versa. Therefore, both the matrices $\bm{A}\bm{B}$ and $\bm{B}\bm{A}$ have a common spectrum, and thus their spectral radii coincide.

\subsection{Vector inequalities}

Given a matrix $\bm{A}\in\mathbb{X}^{m\times n}$ and vector $\bm{d}\in\mathbb{X}^{m}$, consider the problem to find vectors $\bm{x}\in\mathbb{X}^{n}$ to satisfy the inequality
\begin{equation}
\bm{A}\bm{x}
\leq
\bm{d}
\label{I-Axleqd}
\end{equation}

Solutions of the problem are known under various assumptions and in different forms (see, e.\,g., \cite{Baccelli1993Synchronization,Cuninghamegreen1994Minimax}). Below, we use the solution \cite{Krivulin2015Extremal} offered by the next statement.
\begin{proposition}
\label{P-Axleqd}
For any column-regular matrix $\bm{A}$ and regular vector $\bm{d}$, all solutions to inequality \eqref{I-Axleqd} are given by
\begin{equation*}
\bm{x}
\leq
(\bm{d}^{-}\bm{A})^{-}.
\end{equation*}
\end{proposition}

Given a matrix $\bm{A}\in\mathbb{X}^{n\times n}$ and vectors $\bm{b},\bm{d}\in\mathbb{X}^{n}$, consider the problem of finding regular vectors $\bm{x}\in\mathbb{X}^{n}$ that satisfy the double inequality
\begin{equation}
\bm{A}\bm{x}\oplus\bm{b}
\leq
\bm{x}
\leq
\bm{d}
\label{I-Axbleqxleqd}
\end{equation}

The next statement offers a complete solution to the problem \cite{Krivulin2014Constrained}.
\begin{proposition}
\label{P-Axbleqxleqd}
For any matrix $\bm{A}$, vector $\bm{b}$ and regular vector $\bm{d}$, denote $\Delta=\mathop\mathrm{Tr}(\bm{A})\oplus\bm{d}^{-}\bm{A}^{\ast}\bm{b}$. Then, the following statements hold.
\begin{itemize}
\item[(i)]
If $\Delta\leq\mathbb{1}$, then all regular solutions of inequality \eqref{I-Axbleqxleqd} are given by
\begin{equation*}
\bm{x}
=
\bm{A}^{\ast}\bm{u},
\quad
\bm{b}\leq\bm{u}\leq(\bm{d}^{-}\bm{A}^{\ast})^{-}.
\end{equation*}
\item[(ii)]
If $\Delta>\mathbb{1}$, then there is no regular solution.
\end{itemize}
\end{proposition}

\subsection{Binomial identities for traces}
\label{ss:traces}

We now present identities for powers of matrix binomials and their traces, which are used in algebraic manipulations in what follows. First note that, for any square matrices $\bm{A},\bm{B}\in\mathbb{X}^{n\times n}$ and positive integer $k$, the following equality is valid (see also \cite{Krivulin2014Constrained,Krivulin2015Extremal,Krivulin2015Multidimensional,Krivulin2017Direct}):
\begin{equation*}
(\bm{A}\oplus\bm{B})^{k}
=
\bigoplus_{l=1}^{k}
\bigoplus_{i_{0}+i_{1}+\cdots+i_{l}=k-l}
\bm{B}^{i_{0}}(\bm{A}\bm{B}^{i_{1}}\cdots\bm{A}\bm{B}^{i_{l}})
\oplus
\bm{B}^{k}.
\end{equation*}

By summing over all $k=1,\ldots,p$, where $p\geq1$, and rearranging terms, we obtain
\begin{equation}
\bigoplus_{k=1}^{p}
(\bm{A}\oplus\bm{B})^{k}
=
\bigoplus_{k=1}^{p}
\bigoplus_{0\leq i_{0}+i_{1}+\cdots+i_{k}\leq p-k}
\bm{B}^{i_{0}}(\bm{A}\bm{B}^{i_{1}}\cdots\bm{A}\bm{B}^{i_{k}})
\oplus
\bigoplus_{k=1}^{p}
\bm{B}^{k}.
\label{E-oplus1pABk}
\end{equation}

Similarly, summing over all $k=0,\ldots,n-1$ yields
\begin{equation*}
(\bm{A}\oplus\bm{B})^{\ast}
=
\bigoplus_{k=0}^{n-1}
\bigoplus_{0\leq i_{0}+i_{1}+\cdots+i_{k}\leq n-k-1}
\bm{B}^{i_{0}}(\bm{A}\bm{B}^{i_{1}}\cdots\bm{A}\bm{B}^{i_{k}})
\oplus
\bm{B}^{\ast}.
\end{equation*}

After taking trace of both sides of equality \eqref{E-oplus1pABk}, we have the identity
\begin{multline}
\bigoplus_{k=1}^{p}
\mathop\mathrm{tr}(\bm{A}\oplus\bm{B})^{k}
=
\bigoplus_{k=1}^{p}
\bigoplus_{0\leq i_{0}+i_{1}+\cdots+i_{k}\leq p-k}
\mathop\mathrm{tr}(\bm{B}^{i_{0}}(\bm{A}\bm{B}^{i_{1}}\cdots\bm{A}\bm{B}^{i_{k}}))
\oplus
\bigoplus_{k=1}^{p}
\mathop\mathrm{tr}\bm{B}^{k}.
\label{E-oplus1ptrABk}
\end{multline}

Specifically, if $p=n$, then the identity takes the form
\begin{equation*}
\mathop\mathrm{Tr}(\bm{A}\oplus\bm{B})
=
\bigoplus_{k=1}^{n}
\bigoplus_{0\leq i_{0}+i_{1}+\cdots+i_{k}\leq n-k}
\mathop\mathrm{tr}(\bm{B}^{i_{0}}(\bm{A}\bm{B}^{i_{1}}\cdots\bm{A}\bm{B}^{i_{k}}))
\oplus
\mathop\mathrm{Tr}(\bm{B}).
\end{equation*}

We conclude this section with the estimation of the computational complexity associated with the calculation of the sum at \eqref{E-oplus1pABk} (see also \cite{Krivulin2017Direct,Krivulin2017Tropical}). First, we introduce the notation
\begin{equation*}
\bm{T}_{kl}
=
\bigoplus_{0\leq i_{0}+i_{1}+\cdots+i_{k}\leq l}
\bm{B}^{i_{0}}(\bm{A}\bm{B}^{i_{1}}\cdots\bm{A}\bm{B}^{i_{k}}),
\qquad
k,l\geq0,
\end{equation*}
and note that
\begin{equation*}
\bigoplus_{k=1}^{p}
(\bm{A}\oplus\bm{B})^{k}
=
\bigoplus_{k=1}^{p}
\bm{T}_{k,p-k}\oplus\bigoplus_{k=1}^p \bm{B}^k.
\end{equation*}

We prove the following statement.
\begin{proposition}
\label{P:recurrence}
The matrices $\bm{T}_{kl}$ satisfy the recurrence relations
\begin{equation*}
\bm{T}_{kl}
=
\bm{A}\bm{T}_{k-1,l}\oplus\bm{B}\bm{T}_{k,l-1},
\quad
\bm{T}_{k0}
=
\bm{A}^{k},
\quad
\bm{T}_{0l}
=
\bm{B}^{l},
\quad
\bm{T}_{00}
=
\bm{I},
\qquad
k,l\geq0.
\end{equation*}
\end{proposition}
\begin{proof}
To verify the first relation (the others are obvious), we write
\begin{multline*}
\bm{A}\bm{T}_{k-1,l}
=
\bigoplus_{0\leq i_{0}+i_{1}+\cdots+i_{k-1}\leq l}
\bm{B}^{0}\bm{A}\bm{B}^{i_{0}}(\bm{A}\bm{B}^{i_{1}}\cdots\bm{A}\bm{B}^{i_{k-1}})
\\=
\bigoplus_{0\leq i_{1}+\cdots+i_{k}\leq l}
\bm{B}^{0}(\bm{A}\bm{B}^{i_{1}}\cdots\bm{A}\bm{B}^{i_{k}}).
\end{multline*}

Furthermore, we have
\begin{multline*}
\bm{B}\bm{T}_{k,l-1}
=
\bigoplus_{0\leq i_{0}+i_{1}+\cdots+i_{k}\leq l-1}
\bm{B}^{i_{0}+1}(\bm{A}\bm{B}^{i_{1}}\cdots\bm{A}\bm{B}^{i_{k}})
\\=
\bigoplus_{0\leq (i_{0}+1)+i_{1}+\cdots+i_{k}\leq l}
\bm{B}^{i_{0}+1}(\bm{A}\bm{B}^{i_{1}}\cdots\bm{A}\bm{B}^{i_{k}}).
\end{multline*}

By combining both expressions, we obtain the desired recurrence relation.
\qed
\end{proof}

\begin{corollary}
\label{C:complexity}
The calculation of sums \eqref{E-oplus1pABk} and \eqref{E-oplus1ptrABk} involves $O(n^{3}p^{2})$ scalar operations.
\end{corollary}
\begin{proof}
It follows from the above recurrence relation, that each matrix $\bm{T}_{kl}$ can be calculated with one matrix addition and two matrix multiplications, which requires $O(n^{3})$ scalar operations. Observing that the number of matrices $\bm{T}_{kl}$ needed to evaluate the sum at \eqref{E-oplus1pABk} is $1+\cdots+p=p(p+1)/2$, the overall computational complexity is of order $O(n^{3}p^{2})$. Specifically, with $p=n$, we have the order $O(n^{5})$.

Finally note that the calculation of the sum of traces at \eqref{E-oplus1ptrABk} has the same order of complexity $O(n^{3}p^{2})$, which becomes $O(n^{5})$ if $p=n$.
\qed
\end{proof}

\subsection{Skew block diagonal matrices}

Given matrices $\bm{B}\in\mathbb{X}^{n\times m}$ and $\bm{C}\in\mathbb{X}^{m\times n}$, we consider a square matrix of order $m+n$ in the skew block diagonal form
\begin{equation*}
\bm{A}
=
\begin{pmatrix}
\bm{0} & \bm{B}
\\
\bm{C} & \bm{0}
\end{pmatrix}.
\end{equation*}

Here we are going to consider tropical trace functions and Kleene stars of skew block diagonal matrices. The results presented further extend the technique developed in \cite{Carre1971Algebra} to compute Kleene stars of block matrices.
\begin{proposition}
\label{R-TrA}
If $\mathop\mathrm{Tr}(\bm{A})\leq\mathbb{1}$, then the following equalities hold:
\begin{equation*}
\mathop\mathrm{Tr}(\bm{A})
=
\mathop\mathrm{Tr}(\bm{B}\bm{C})
=
\mathop\mathrm{Tr}(\bm{C}\bm{B})
=
\bigoplus_{k=1}^{\min(m,n)}\mathop\mathrm{tr}(\bm{B}\bm{C})^{k}
=
\bigoplus_{k=1}^{\min(m,n)}\mathop\mathrm{tr}(\bm{C}\bm{B})^{k}.
\end{equation*}
\end{proposition}
\begin{proof}
First assume that $m\leq n$. Since $\mathop\mathrm{Tr}(\bm{A})\leq\mathbb{1}$, the inequality $\bm{A}^{k}\leq\bm{A}^{\ast}$ holds for any integer $k\geq0$. Observing that $m+n\leq 2n$, we have
\begin{equation*}
\bm{A}^{\ast}
=
\bigoplus_{k=0}^{m+n-1}\bm{A}^{k}
=
\bigoplus_{k=0}^{2n-1}\bm{A}^{k},
\qquad
\bigoplus_{k=1}^{m+n}\bm{A}^{k}
=
\bm{A}\bm{A}^{\ast}
=
\bigoplus_{k=1}^{2n}\bm{A}^{k}.
\end{equation*}

Next, we calculate
\begin{equation*}
\mathop\mathrm{Tr}(\bm{A})
=
\bigoplus_{k=1}^{m+n}\mathop\mathrm{tr}\bm{A}^{k}
=
\bigoplus_{k=1}^{2n}\mathop\mathrm{tr}\bm{A}^{k}
=
\bigoplus_{k=1}^{n}\mathop\mathrm{tr}\bm{A}^{2k}.
\end{equation*}

For any $k=0,1,2,\ldots$, we obtain
\begin{equation*}
\bm{A}^{2k}
=
\begin{pmatrix}
(\bm{B}\bm{C})^{k} & \bm{0}
\\
\bm{0} & (\bm{C}\bm{B})^{k}
\end{pmatrix},
\qquad
\bm{A}^{2k+1}
=
\begin{pmatrix}
\bm{0} & \bm{B}(\bm{C}\bm{B})^{k}
\\
\bm{C}(\bm{B}\bm{C})^{k} & \bm{0}
\end{pmatrix},
\end{equation*}
and hence $\mathop\mathrm{tr}\bm{A}^{2k}=\mathop\mathrm{tr}(\bm{B}\bm{C})^{k}=\mathop\mathrm{tr}(\bm{C}\bm{B})^{k}$ and $\mathop\mathrm{tr}\bm{A}^{2k+1}=\mathbb{0}$.

We now can write
\begin{equation*}
\mathop\mathrm{Tr}(\bm{A})
=
\bigoplus_{k=1}^{n}\mathop\mathrm{tr}(\bm{C}\bm{B})^{k}
=
\bigoplus_{k=1}^{n}\mathop\mathrm{tr}(\bm{B}\bm{C})^{k}
=
\mathop\mathrm{Tr}(\bm{B}\bm{C}).
\end{equation*}

At the same time, we have
\begin{equation*}
\mathop\mathrm{Tr}(\bm{C}\bm{B})
=
\bigoplus_{k=1}^{m}\mathop\mathrm{tr}(\bm{C}\bm{B})^{k}
\leq
\mathop\mathrm{Tr}(\bm{A}).
\end{equation*}

Since $\mathop\mathrm{Tr}(\bm{C}\bm{B})\leq\mathop\mathrm{Tr}(\bm{A})\leq\mathbb{1}$, we obtain
\begin{equation*}
\bigoplus_{k=1}^{m}(\bm{C}\bm{B})^{k}
=
(\bm{C}\bm{B})(\bm{C}\bm{B})^{\ast}
=
\bigoplus_{k=1}^{n}(\bm{C}\bm{B})^{k}.
\end{equation*}

Turning to traces yields
\begin{equation*}
\mathop\mathrm{Tr}(\bm{C}\bm{B})
=
\bigoplus_{k=1}^{n}\mathop\mathrm{tr}(\bm{C}\bm{B})^{k}
=
\bigoplus_{k=1}^{n}\mathop\mathrm{tr}(\bm{B}\bm{C})^{k}
=
\mathop\mathrm{Tr}(\bm{B}\bm{C})
=
\mathop\mathrm{Tr}(\bm{A}).
\end{equation*}

The case when $m>n$ is examined in the same way.
\qed
\end{proof}

\begin{proposition}
\label{R-Aast}
If $\mathop\mathrm{Tr}(\bm{A})\leq\mathbb{1}$, then the following equalities hold:
\begin{equation*}
\bm{A}^{\ast}
=
\begin{pmatrix}
(\bm{B}\bm{C})^{\ast} & \bm{B}(\bm{C}\bm{B})^{\ast}
\\
\bm{C}(\bm{B}\bm{C})^{\ast} & (\bm{C}\bm{B})^{\ast}
\end{pmatrix}
=
\bigoplus_{k=0}^{\min(m,n)}
\begin{pmatrix}
(\bm{B}\bm{C})^{k} & \bm{B}(\bm{C}\bm{B})^{k}
\\
\bm{C}(\bm{B}\bm{C})^{k} & (\bm{C}\bm{B})^{k}
\end{pmatrix}.
\end{equation*}
\end{proposition}
\begin{proof}
Consider the matrix $\bm{A}^{\ast}$ under assumption that $m\leq n$, and represent this matrix as
\begin{equation*}
\bm{A}^{\ast}
=
\bigoplus_{k=0}^{m+n-1}\bm{A}^{k}
=
\bigoplus_{k=0}^{2n-1}\bm{A}^{k}
=
\bigoplus_{k=0}^{n-1}
\bm{A}^{2k}
\oplus
\bigoplus_{k=0}^{n-1}
\bm{A}^{2k+1}.
\end{equation*}

Substitution of the even and odd powers of the matrix $\bm{A}$ yields
\begin{equation*}
\bm{A}^{\ast}
=
\bigoplus_{k=0}^{n-1}
\begin{pmatrix}
(\bm{B}\bm{C})^{k} & \bm{B}(\bm{C}\bm{B})^{k}
\\
\bm{C}(\bm{B}\bm{C})^{k} & (\bm{C}\bm{B})^{k}
\end{pmatrix}.
\end{equation*}

Note that, under the condition that $m=n$, we immediately have
\begin{equation*}
\bm{A}^{\ast}
=
\begin{pmatrix}
(\bm{B}\bm{C})^{\ast} & \bm{B}(\bm{C}\bm{B})^{\ast}
\\
\bm{C}(\bm{B}\bm{C})^{\ast} & (\bm{C}\bm{B})^{\ast}
\end{pmatrix}.
\end{equation*}

We now assume that $m<n$ and consider the blocks in the matrix $\bm{A}^{\ast}$. First, we note that the inequalities $(\bm{B}\bm{C})^{k}\leq(\bm{B}\bm{C})^{\ast}$ and $(\bm{C}\bm{B})^{k}\leq(\bm{C}\bm{B})^{\ast}$ hold for all integer $k\geq0$ since $\mathop\mathrm{Tr}(\bm{B}\bm{C})=\mathop\mathrm{Tr}(\bm{C}\bm{B})=\mathop\mathrm{Tr}(\bm{A})\leq\mathbb{1}$.

Observing that $m-1\leq n-2$, we represent the upper left block as
\begin{equation*}
\bigoplus_{k=0}^{n-1}
(\bm{B}\bm{C})^{k}
=
\bm{I}
\oplus
\bm{B}
\bigoplus_{k=0}^{n-2}
(\bm{C}\bm{B})^{k}
\bm{C}
=
\bm{I}
\oplus
\bm{B}
(\bm{C}\bm{B})^{\ast}
\bm{C}
=
\bigoplus_{k=0}^{m}
(\bm{B}\bm{C})^{k}.
\end{equation*}

Furthermore, for the upper right block, we have
\begin{equation*}
\bigoplus_{k=0}^{n-1}
\bm{B}(\bm{C}\bm{B})^{k}
=
\bm{B}
(\bm{C}\bm{B})^{\ast}
=
\bigoplus_{k=0}^{m}
\bm{B}(\bm{C}\bm{B})^{k}.
\end{equation*}

Similarly, we calculate the rest two blocks to obtain
\begin{equation*}
\bigoplus_{k=0}^{n-1}
\bm{C}(\bm{B}\bm{C})^{k}
=
\bigoplus_{k=0}^{m}
\bm{C}(\bm{B}\bm{C})^{k},
\qquad
\bigoplus_{k=0}^{n-1}
(\bm{C}\bm{B})^{k}
=
\bigoplus_{k=0}^{m}
(\bm{C}\bm{B})^{k}.
\end{equation*}

Under the condition that $m>n$, the proof follows the same argument.
\qed
\end{proof}

\begin{corollary}
If $\mathop\mathrm{Tr}(\bm{A})\leq\mathbb{1}$ and $m=n$, then
\begin{equation*}
\bm{A}^{\ast}
=
\bigoplus_{k=0}^{n-1}
\begin{pmatrix}
(\bm{B}\bm{C})^{k} & \bm{B}(\bm{C}\bm{B})^{k}
\\
\bm{C}(\bm{B}\bm{C})^{k} & (\bm{C}\bm{B})^{k}
\end{pmatrix}.
\end{equation*}
\end{corollary}

\section{Solution of two-stage scheduling problem}
\label{s:solution}

Let us now formulate the two-stage problem at \eqref{P-Bi-Level-Simple} in terms of tropical algebra. Using operations of max-plus algebra, we rewrite the problem as
\begin{equation*}
\begin{aligned}
\min_{\bm{x},\bm{y}\; \text{regular}}
&&&
\bigoplus_{1\leq i\leq m}
y_{i}^{-1}\left(\bigoplus_{1\leq j\leq n}a_{ij}x_{j}\right);
\\
\text{s.t.}
&&&
\bigoplus_{1\leq i\leq m}b_{ji}^{-1}y_{i}
\leq
x_{j},
\quad
j=1,\ldots,n;
\\
&&&
(\bm{x},\bm{y})
\in
\arg\min_{\bm{u},\bm{v}\; \text{regular}}
\bigg\{
\bigoplus_{1\leq i\leq m}v_{i}^{-1}\bigg(\bigoplus_{1\leq j\leq n}c_{ij}u_{j}\bigg):
\\
&&&\qquad\qquad
\bigoplus_{1\leq i\leq m}d_{ji}^{-1}v_{i}
\leq
u_{j},
\quad
g_{j}\leq u_{j}\leq h_{j},
\quad
j=1,\ldots,n;
\\
&&&\qquad\qquad
q_{i}\leq v_{i}\leq r_{i},
\quad
i=1,\ldots,m
\bigg\}.
\end{aligned}
\end{equation*}

In addition to the vectors $\bm{y}=(y_{i})$, $\bm{x}=(x_{j})$, $\bm{v}=(v_{i})$ and $\bm{u}=(u_{j})$, we introduce the $(m\times n)$-matrices $\bm{A}=(a_{ij})$, $\bm{B}=(b_{ij})$, $\bm{C}=(c_{ij})$ and $\bm{D}=(d_{ij})$. Furthermore, we introduce the $m$-vectors $\bm{q}=(q_{i})$ and $\bm{r}=(r_{i})$, and $n$-vectors $\bm{g}=(g_{j})$ and $\bm{h}=(h_{j})$.

With this notation, the two-stage problem takes the vector form
\begin{equation}
\begin{aligned}
\min_{\bm{x},\bm{y}\; \text{regular}}
&&&
\bm{y}^{-}\bm{A}\bm{x};
\\
\text{s.t.}
&&&
\bm{B}^{-}\bm{y}
\leq
\bm{x};
\\
&&&
(\bm{x},\bm{y})
\in
\arg\min_{\bm{u},\bm{v}}
\bigg\{
\bm{v}^{-}\bm{C}\bm{u}:
\\
&&&\qquad\qquad\qquad
\bm{D}^{-}\bm{v}
\leq
\bm{u},
\quad
\bm{g}\leq\bm{u}\leq\bm{h},
\quad
\bm{q}\leq\bm{v}\leq\bm{r}
\bigg\}.
\end{aligned}
\label{P-Bi-Level}
\end{equation}

To solve both first-stage and second-stage problems, we apply a solution technique that involves introducing a parameter to represent the value of the objective function in the problem, followed by reducing the problem to a parametrized inequality \cite{Krivulin2014Constrained,Krivulin2015Extremal,Krivulin2015Multidimensional}. Then, the existence conditions for solutions of the inequality are used to derive the minimal value of the parameter. A complete solution of the minimization problem is obtained as the solution of the inequality, which corresponds to the minimum value.

\subsection{Solution of first-stage problem}
\label{ss:lower-level}

Consider the first-stage problem at \eqref{P-Bi-Level} in the vector form
\begin{equation}
\begin{aligned}
\min_{\bm{u},\bm{v}\; \text{regular}}
&&&
\bm{v}^{-}\bm{C}\bm{u};
\\
\text{s.t.}
&&&
\bm{D}^{-}\bm{v}
\leq
\bm{u},
\quad
\bm{g}\leq\bm{u}\leq\bm{h},
\quad
\bm{q}\leq\bm{v}\leq\bm{r}.
\end{aligned}
\label{P:lower-level}
\end{equation}

The next result provides the minimum value of the objective function, and offers a system of inequalities that define the solution set of the problem.
\begin{lemma}
\label{L:lower-level}
Problem \eqref{P:lower-level} with regular vectors $\bm{h}$ and $\bm{r}$ and a nonzero matrix $\bm{C}$ is feasible if and only if $\bm{h}^{-}\bm{g}\oplus(\bm{h}^{-}\bm{D}^{-}\oplus\bm{r}^{-})\bm{q}\leq
\mathbb{1}$.

Under these conditions the optimal value of \eqref{P:lower-level} is greater than $\mathbb{0}$ and equal to
\begin{multline}
\mu
=
\rho(\bm{C}\bm{D}^{-})
\oplus
\bigoplus_{k=1}^{\min(m,n)}
\left(
\bm{h}^{-}(\bm{D}^{-}\bm{C})^{k}\bm{g}
\oplus
(\bm{h}^{-}\bm{D}^{-}
\oplus
\bm{r}^{-})(\bm{C}\bm{D}^{-})^{k}\bm{q}
\right)^{1/k}
\\\oplus
\bigoplus_{k=0}^{\min(m,n)}
(\bm{r}^{-}\bm{C}(\bm{D}^{-}\bm{C})^{k}\bm{g})^{1/(k+1)}.
\label{E:mu}
\end{multline}

Regular vectors $\bm{u}$ and $\bm{v}$ are solutions of \eqref{P:lower-level} if and only if they satisfy the following system of inequalities:
\begin{equation}
\begin{aligned}
\mu^{-1}\bm{C}\bm{u}\oplus\bm{q}
\leq
\bm{v}
\leq
\bm{r},
\\
\bm{D}^{-}\bm{v}\oplus\bm{g}
\leq
\bm{u}
\leq
\bm{h}.
\end{aligned}
\label{e-LowLevelSystem}
\end{equation}
\end{lemma}
\begin{proof}
To solve the problem obtained, we first define an auxiliary parameter $\theta$ to represent the problem as follows:
\begin{equation*}
\begin{aligned}
\min_{\bm{u},\bm{v}\; \text{regular}}
&&&
\theta;
\\
\text{s.t.}
&&&
\bm{v}^{-}\bm{C}\bm{u}
\leq
\theta,
\\
&&&
\bm{D}^{-}\bm{v}
\leq
\bm{u},
\quad
\bm{g}\leq\bm{u}\leq\bm{h},
\quad
\bm{q}\leq\bm{v}\leq\bm{r}.
\end{aligned}
\end{equation*}

If the vectors $\bm{r}$ and $\bm{h}$ are regular and $\bm{C}\ne\bm{0}$, then the function $\bm{v}^{-}\bm{C}\bm{u}$ with regular $\bm{v}$ and $\bm{u}$ is bounded from below, and the optimal value of $\theta$ is greater than $\mathbb{0}$.

Furthermore, we rearrange the inequality constraints in the problem. We apply Proposition~\ref{P-Axleqd} to solve the first inequality with respect to $\bm{C}\bm{u}$, and then rewrite the result as $\theta^{-1}\bm{C}\bm{u}\leq\bm{v}$. Next, by coupling the inequality obtained with the last inequality constraint, and the second inequality with the third, we represent the problem in the form
\begin{equation}
\begin{aligned}
\min_{\bm{u},\bm{v}\; \text{regular}}
&&&
\theta;
\\
\text{s.t.}
&&&
\theta^{-1}\bm{C}\bm{u}\oplus\bm{q}
\leq
\bm{v}
\leq
\bm{r},
\\
&&&
\bm{D}^{-}\bm{v}\oplus\bm{g}
\leq
\bm{u}
\leq
\bm{h}.
\end{aligned}
\label{P-minuvtheta-theta1Cugleqvleqh-Dvqlequleqr}
\end{equation}

We define the following skew block diagonal matrix and vectors:
\begin{equation*}
\bm{F}_{\theta}
=
\begin{pmatrix}
\bm{0} & \bm{D}^{-}
\\
\theta^{-1}\bm{C} & \bm{0}
\end{pmatrix},
\quad
\bm{s}
=
\begin{pmatrix}
\bm{g}
\\
\bm{q}
\end{pmatrix},
\quad
\bm{t}
=
\begin{pmatrix}
\bm{h}
\\
\bm{r}
\end{pmatrix},
\quad
\bm{z}
=
\begin{pmatrix}
\bm{u}
\\
\bm{v}
\end{pmatrix},
\end{equation*}
and note that the matrix $\bm{F}_{\theta}$ is of order $m+n$, and the vectors $\bm{s}$, $\bm{t}$ and $\bm{z}$ are of order $m+n$. With this notation, the problem becomes
\begin{equation}
\begin{aligned}
\min_{\bm{z}\; \text{regular}}
&&&
\theta;
\\
\text{s.t.}
&&&
\bm{F}_{\theta}\bm{z}\oplus\bm{s}
\leq
\bm{z}
\leq
\bm{t}.
\end{aligned}
\label{P-minztheta-Fxsleqzleqt}
\end{equation}

We now use the existence condition for regular solutions from Proposition~\ref{P-Axbleqxleqd} to define the feasible set for the parameter $\theta$ in problem \eqref{P-minztheta-Fxsleqzleqt}. The condition takes the form of the inequality $\mathop\mathrm{Tr}(\bm{F}_{\theta})\oplus\bm{t}^{-}\bm{F}_{\theta}^{\ast}\bm{s}\leq\mathbb{1}$, and thus problem \eqref{P-minztheta-Fxsleqzleqt} reduces to finding the minimum value of the parameter $\theta$, over all $\theta$ that satisfy this inequality.

To solve this inequality, we note that it is equivalent to two inequalities
\begin{equation*}
\mathop\mathrm{Tr}(\bm{F}_{\theta})\leq\mathbb{1},
\qquad
\bm{t}^{-}\bm{F}_{\theta}^{\ast}\bm{s}\leq\mathbb{1}.
\end{equation*}

Consider the first inequality, and assume that $m\leq n$. Observing that the matrix $\bm{C}\bm{D}^{-}$ is of order $m$, we apply Proposition~\ref{R-TrA} to write
\begin{equation*}
\mathop\mathrm{Tr}(\bm{F}_{\theta})
=
\bigoplus_{k=1}^{m}\theta^{-k}\mathop\mathrm{tr}(\bm{C}\bm{D}^{-})^{k}
\leq
\mathbb{1}.
\end{equation*}

This inequality is equivalent to the system of inequalities
\begin{equation*}
\theta^{-k}\mathop\mathrm{tr}(\bm{C}\bm{D}^{-})^{k}
\leq
\mathbb{1},
\quad
k=1,\ldots,m,
\end{equation*}
which can be solved for $\theta$ as follows:
\begin{equation*}
\mathop\mathrm{tr}\nolimits^{1/k}((\bm{C}\bm{D}^{-})^{k})
\leq
\theta,
\quad
k=1,\ldots,m.
\end{equation*}

Combining the inequalities into one yields the inequality
\begin{equation*}
\theta
\geq
\bigoplus_{k=1}^{m}
\mathop\mathrm{tr}\nolimits^{1/k}((\bm{C}\bm{D}^{-})^{k})
=
\rho(\bm{C}\bm{D}^{-})
=
\rho(\bm{D}^{-}\bm{C}),
\end{equation*}
which does not change if the condition $m\leq n$ is replaced by $m>n$.

To examine the second inequality $\bm{t}^{-}\bm{F}_{\theta}^{\ast}\bm{s}\leq\mathbb{1}$, we first use Proposition~\ref{R-Aast} to obtain the matrix
\begin{equation*}
\bm{F}_{\theta}^{\ast}
=
\bigoplus_{k=0}^{\min(m,n)}
\begin{pmatrix}
(\theta^{-1}\bm{D}^{-}\bm{C})^{k} & \bm{D}^{-}(\theta^{-1}\bm{C}\bm{D}^{-})^{k}
\\
\theta^{-1}\bm{C}(\theta^{-1}\bm{D}^{-}\bm{C})^{k} & (\theta^{-1}\bm{C}\bm{D}^{-})^{k}
\end{pmatrix}.
\end{equation*}

After substitution of the matrix $\bm{F}_{\theta}^{\ast}$ and vectors $\bm{t}$ and $\bm{s}$, the inequality becomes
\begin{multline*}
\bigoplus_{k=0}^{\min(m,n)}
\theta^{-k}
\left(
\bm{h}^{-}(\bm{D}^{-}\bm{C})^{k}\bm{g}
\oplus
(\bm{h}^{-}\bm{D}^{-}\oplus\bm{r}^{-})
(\bm{C}\bm{D}^{-})^{k}\bm{q}
\right)
\\\oplus
\bigoplus_{k=0}^{\min(m,n)}
\theta^{-k-1}
\bm{r}^{-}\bm{C}(\bm{D}^{-}\bm{C})^{k}\bm{g}
\leq
\mathbb{1},
\end{multline*}
which can be represented as the system of inequalities
\begin{align*}
\bm{h}^{-}\bm{g}
\oplus
(\bm{h}^{-}\bm{D}^{-}
\oplus
\bm{r}^{-})\bm{q}
&\leq
\mathbb{1},
\\
\bigoplus_{k=1}^{\min(m,n)}
\theta^{-k}
\left(
\bm{h}^{-}(\bm{D}^{-}\bm{C})^{k}\bm{g}
\oplus
(\bm{h}^{-}\bm{D}^{-}
\oplus
\bm{r}^{-})(\bm{C}\bm{D}^{-})^{k}\bm{q}
\right)
&\leq
\mathbb{1},
\\
\bigoplus_{k=0}^{\min(m,n)}
\theta^{-(k+1)}\bm{r}^{-}\bm{C}(\bm{D}^{-}\bm{C})^{k}\bm{g}
&\leq
\mathbb{1}.
\end{align*}

The first inequality in the system are valid by the assumption of the lemma. By solving the last two inequalities with respect to $\theta$ in the same way as above, we obtain
\begin{align*}
\theta
&\geq
\bigoplus_{k=1}^{\min(m,n)}
\left(
\bm{h}^{-}(\bm{D}^{-}\bm{C})^{k}\bm{g}
\oplus
(\bm{h}^{-}\bm{D}^{-}
\oplus
\bm{r}^{-})(\bm{C}\bm{D}^{-})^{k}\bm{q}
\right)^{1/k},
\\
\theta
&\geq
\bigoplus_{k=0}^{\min(m,n)}
(\bm{r}^{-}\bm{C}(\bm{D}^{-}\bm{C})^{k}\bm{g})^{1/(k+1)}.
\end{align*}

By combining all lower bounds for $\theta$, we have
\begin{multline*}
\theta
\geq
\rho(\bm{C}\bm{D}^{-})
\oplus
\bigoplus_{k=1}^{\min(m,n)}
\left(
\bm{h}^{-}(\bm{D}^{-}\bm{C})^{k}\bm{g}
\oplus
(\bm{h}^{-}\bm{D}^{-}
\oplus
\bm{r}^{-})(\bm{C}\bm{D}^{-})^{k}\bm{q}
\right)^{1/k}
\\\oplus
\bigoplus_{k=0}^{\min(m,n)}
(\bm{r}^{-}\bm{C}(\bm{D}^{-}\bm{C})^{k}\bm{g})^{1/(k+1)}.
\end{multline*}

Setting the minimum $\mu$ in the problem to be equal to the combined bound yields \eqref{E:mu}. Finally, by replacing $\theta$ by $\mu$ in the system of inequality constraints of \eqref{P-minuvtheta-theta1Cugleqvleqh-Dvqlequleqr}, we obtain the system of inequalities at \eqref{e-LowLevelSystem}.
\qed
\end{proof}

\subsection{Solution of second-stage problem}
\label{ss:upperlevel}

Consider the second-stage problem, and adjust the description of the feasible solution set by adding the constraints elevated from the solution of the first-stage problem. With the new constraints  obtained by replacing $\bm{u}$ and $\bm{v}$ by $\bm{x}$ and $\bm{y}$ in the system at \eqref{e-LowLevelSystem}, we have the problem
\begin{equation}
\begin{aligned}
\min_{\bm{x},\bm{y}\; \text{regular}}
&&&
\bm{y}^{-}\bm{A}\bm{x};
\\
\text{s.t.}
&&&
\bm{B}^{-}\bm{y}
\leq
\bm{x},
\\
&&&
\mu^{-1}\bm{C}\bm{x}\oplus\bm{q}
\leq
\bm{y}
\leq
\bm{r},
\\
&&&
\bm{D}^{-}\bm{y}\oplus\bm{g}
\leq
\bm{x}
\leq
\bm{h}.
\end{aligned}
\label{MainProblem2}
\end{equation}
where $\mu$ is given by \eqref{E:mu}.

To describe a solution of this problem, we define the matrices
\begin{equation*}
\bm{D}_{1}^{-}
=
\bm{B}^{-}\oplus\bm{D}^{-},
\qquad
\bm{C}_{1}
=
\mu^{-1}\bm{C},
\end{equation*}
and then introduce the notation
\begin{equation*}
\bm{P}
=
\bm{A}\bm{D}_{1}^{-},
\qquad
\bm{Q}
=
\bm{C}_{1}\bm{D}_{1}^{-},
\qquad
\bm{R}
=
\bm{D}_{1}^{-}\bm{A},
\qquad
\bm{S}
=
\bm{D}_{1}^{-}\bm{C}_{1}.
\end{equation*}

\begin{theorem}
\label{T:upper-level}
Suppose that $\bm{h}$ and $\bm{r}$ are regular vectors and $\bm{A}$ is a nonzero matrix. Then, problem \eqref{MainProblem2} is feasible if and only if
\begin{equation*}
\mathop\mathrm{Tr}(\bm{Q})
\oplus
(\bm{h}^{-}\bm{D}_{1}^{-}
\oplus
\bm{r}^{-})\bm{Q}^{\ast}\bm{q}
\oplus
(\bm{r}^{-}\bm{C}_{1}
\oplus
\bm{h}^{-})\bm{S}^{\ast}\bm{g}
\leq
\mathbb{1}.
\end{equation*}

In this case, the optimal value of \eqref{MainProblem2} is greater than $\mathbb{0}$ and equal to
\begin{multline}
\eta
=
\bigoplus_{k=1}^{\min(m,n)}
\bigoplus_{0\leq i_{0}+i_{1}+\cdots+i_{k}\leq\min(m,n)-k}
\mathop\mathrm{tr}\nolimits^{1/k}
(\bm{Q}^{i_{0}}(\bm{P}\bm{Q}^{i_{1}}\cdots\bm{P}\bm{Q}^{i_{k}}))
\\\oplus
\bigoplus_{k=1}^{\min(m,n)}
\bigoplus_{0\leq i_{0}+i_{1}+\cdots+i_{k}\leq \min(m,n)-k}
((\bm{r}^{-}\bm{C}_{1}
\oplus
\bm{h}^{-})
\bm{S}^{i_{0}}(\bm{R}\bm{S}^{i_{1}}\cdots\bm{R}\bm{S}^{i_{k}})
\bm{g})^{1/k}
\\\oplus
\bigoplus_{k=1}^{\min(m,n)}
\bigoplus_{0\leq i_{0}+i_{1}+\cdots+i_{k}\leq \min(m,n)-k}
((\bm{h}^{-}\bm{D}_{1}^{-}
\oplus
\bm{r}^{-})
\bm{Q}^{i_{0}}(\bm{P}\bm{Q}^{i_{1}}\cdots\bm{P}\bm{Q}^{i_{k}})
\bm{q})^{1/k}
\\\oplus
\bigoplus_{k=0}^{\min(m,n)}
\bigoplus_{0\leq i_{0}+i_{1}+\cdots+i_{k}\leq \min(m,n)-k}
(\bm{r}^{-}\bm{A}
\bm{S}^{i_{0}}(\bm{R}\bm{S}^{i_{1}}\cdots\bm{R}\bm{S}^{i_{k}})
\bm{g})^{1/(k+1)}.
\label{e:eta}
\end{multline}

All regular solutions of the problem are given by
\begin{align}
\bm{x}
&=
(\eta^{-1}\bm{R}\oplus\bm{S})^{\ast}
(\bm{u}
\oplus
\bm{D}_{1}^{-}\bm{v}),
\nonumber\\
\bm{y}
&=
(\eta^{-1}\bm{P}\oplus\bm{Q})^{\ast}
((\eta^{-1}\bm{A}\oplus\bm{C}_{1})\bm{u}
\oplus
\bm{v}),
\label{e:xy}
\end{align}
where $\bm{u}$ and $\bm{v}$ are any regular vectors that satisfy the conditions
\begin{align}
\bm{g}
&\leq
\bm{u}
\leq
((\bm{h}^{-}
\oplus
\bm{r}^{-}
(\eta^{-1}\bm{A}\oplus\bm{C}_{1}))(\eta^{-1}\bm{R}\oplus\bm{S})^{\ast})^{-},
\nonumber\\
\bm{q}
&\leq
\bm{v}
\leq
((\bm{h}^{-}
\bm{D}_{1}^{-}
\oplus
\bm{r}^{-})
(\eta^{-1}\bm{P}\oplus\bm{Q})^{\ast})^{-}.
\label{e:uv}
\end{align}
\end{theorem}
\begin{proof}
As in the proof of the previous lemma, we use an additional parameter $\theta$ and combine inequality constraints to replace \eqref{MainProblem2} by the problem
\begin{equation*}
\begin{aligned}
\min_{\bm{x},\bm{y}\; \text{regular}}
&&&
\theta;
\\
\text{s.t.}
&&&
(\theta^{-1}\bm{A}
\oplus
\mu^{-1}\bm{C})\bm{x}
\oplus
\bm{q}
\leq
\bm{y}
\leq
\bm{r},
\\
&&&
(\bm{B}^{-}
\oplus
\bm{D}^{-})\bm{y}
\oplus
\bm{g}
\leq
\bm{x}
\leq
\bm{h}.
\end{aligned}
\end{equation*}

Furthermore, we introduce the notation
\begin{equation*}
\bm{C}_{\theta}
=
\theta^{-1}\bm{A}\oplus\mu^{-1}\bm{C}
=
\theta^{-1}\bm{A}\oplus\bm{C}_{1},
\end{equation*}
and note that, with $\bm{D}_{1}^{-}=\bm{B}^{-}\oplus\bm{D}^{-}$, we have
\begin{equation*}
\bm{D}_{1}^{-}\bm{C}_{\theta}
=
\theta^{-1}\bm{R}\oplus\bm{S},
\qquad
\bm{C}_{\theta}\bm{D}_{1}^{-}
=
\theta^{-1}\bm{P}\oplus\bm{Q}.
\end{equation*}

The problem now becomes
\begin{equation*}
\begin{aligned}
\min_{\bm{x},\bm{y}\; \text{regular}}
&&&
\theta;
\\
\text{s.t.}
&&&
\bm{C}_{\theta}\bm{x}
\oplus
\bm{q}
\leq
\bm{y}
\leq
\bm{r},
\\
&&&
\bm{D}_{1}^{-}\bm{y}
\oplus
\bm{g}
\leq
\bm{x}
\leq
\bm{h}.
\end{aligned}
\end{equation*}

With the matrix-vector notation
\begin{equation*}
\bm{G}_{\theta}
=
\begin{pmatrix}
\bm{0} & \bm{D}_{1}^{-}
\\
\bm{C}_{\theta} & \bm{0}
\end{pmatrix},
\quad
\bm{s}
=
\begin{pmatrix}
\bm{g}
\\
\bm{q}
\end{pmatrix},
\quad
\bm{t}
=
\begin{pmatrix}
\bm{h}
\\
\bm{r}
\end{pmatrix},
\quad
\bm{z}
=
\begin{pmatrix}
\bm{x}
\\
\bm{y}
\end{pmatrix},
\end{equation*}
the problem can be rewritten in the form
\begin{equation*}
\begin{aligned}
\min_{\bm{z}\; \text{regular}}
&&&
\theta;
\\
\text{s.t.}
&&&
\bm{G}_{\theta}\bm{z}\oplus\bm{s}
\leq
\bm{z}
\leq
\bm{t}.
\end{aligned}
\end{equation*}

It follows from Proposition~\ref{P-Axbleqxleqd} that the inequality constraint has regular solutions if and only if the condition $\mathop\mathrm{Tr}(\bm{G}_{\theta})\oplus\bm{t}^{-}\bm{G}_{\theta}^{\ast}\bm{s}\leq\mathbb{1}$ holds, and for any feasible values of the parameter $\theta$, all solutions are given by
\begin{equation}
\bm{z}
=
\bm{G}_{\theta}^{\ast}\bm{w},
\quad
\bm{s}\leq\bm{w}\leq(\bm{t}^{-}\bm{G}_{\theta}^{\ast})^{-},
\quad
\bm{w}
=
\begin{pmatrix}
\bm{u}
\\
\bm{v}
\end{pmatrix}.
\label{E-z}
\end{equation}

We solve the inequality $\mathop\mathrm{Tr}(\bm{G}_{\theta})\oplus\bm{t}^{-}\bm{G}_{\theta}^{\ast}\bm{s}\leq\mathbb{1}$ for the parameter $\theta$, and then find the minimum of $\theta$. We replace the inequality by the equivalent pair of two inequalities
\begin{equation*}
\mathop\mathrm{Tr}(\bm{G}_{\theta})
\leq
\mathbb{1},
\qquad
\bm{t}^{-}\bm{G}_{\theta}^{\ast}\bm{s}
\leq
\mathbb{1}.
\end{equation*}

Consider the first inequality and apply Proposition~\ref{R-TrA} to obtain
\begin{equation*}
\mathop\mathrm{Tr}(\bm{G}_{\theta})
=
\mathop\mathrm{Tr}(\bm{C}_{\theta}\bm{D}_{1}^{-})
=
\mathop\mathrm{Tr}(\theta^{-1}\bm{P}\oplus\bm{Q})
=
\bigoplus_{k=1}^{\min(m,n)}
\mathop\mathrm{tr}(\theta^{-1}\bm{P}\oplus\bm{Q})^{k}
\leq
\mathbb{1}.
\end{equation*}

Furthermore, we use identity \eqref{E-oplus1ptrABk} to rewrite the inequality as
\begin{equation*}
\bigoplus_{k=1}^{\min(m,n)}
\bigoplus_{0\leq i_{0}+i_{1}+\cdots+i_{k}\leq\min(m,n)-k}
\theta^{-k}
\mathop\mathrm{tr}
(\bm{Q}^{i_{0}}(\bm{P}\bm{Q}^{i_{1}}\cdots\bm{P}\bm{Q}^{i_{k}}))
\oplus
\mathop\mathrm{Tr}
(\bm{Q})
\leq
\mathbb{1}.
\end{equation*}

Since $\mathop\mathrm{Tr}(\bm{Q})\leq\mathbb{1}$ by assumption, we need to solve for $\theta$ the inequality
\begin{equation*}
\bigoplus_{k=1}^{\min(m,n)}
\bigoplus_{0\leq i_{0}+i_{1}+\cdots+i_{k}\leq\min(m,n)-k}
\theta^{-k}
\mathop\mathrm{tr}
(\bm{Q}^{i_{0}}(\bm{P}\bm{Q}^{i_{1}}\cdots\bm{P}\bm{Q}^{i_{k}}))
\leq
\mathbb{1}.
\end{equation*}

The solution takes the form
\begin{equation*}
\theta
\geq
\bigoplus_{k=1}^{\min(m,n)}
\bigoplus_{0\leq i_{0}+i_{1}+\cdots+i_{k}\leq\min(m,n)-k}
\mathop\mathrm{tr}\nolimits^{1/k}
(\bm{Q}^{i_{0}}(\bm{P}\bm{Q}^{i_{1}}\cdots\bm{P}\bm{Q}^{i_{k}})).
\end{equation*}

To solve the inequality $\bm{t}^{-}\bm{G}_{\theta}^{\ast}\bm{s}\leq\mathbb{1}$, we apply Proposition~\ref{R-Aast} to write
\begin{equation*}
\bm{G}_{\theta}^{\ast}
\\=
\bigoplus_{k=0}^{\min(m,n)}
\begin{pmatrix}
(\bm{D}_{1}^{-}\bm{C}_{\theta})^{k} & \bm{D}_{1}^{-}(\bm{C}_{\theta}\bm{D}_{1}^{-})^{k}
\\
\bm{C}_{\theta}(\bm{D}_{1}^{-}\bm{C}_{\theta})^{k} & (\bm{C}_{\theta}\bm{D}_{1}^{-})^{k}
\end{pmatrix}.
\end{equation*}

Observing that $\bm{C}_{\theta}(\bm{D}_{1}^{-}\bm{C}_{\theta})^{k}=\theta^{-1}\bm{A}(\bm{D}_{1}^{-}\bm{C}_{\theta})^{k}\oplus\bm{C}_{1}(\bm{D}_{1}^{-}\bm{C}_{\theta})^{k}$, we obtain
\begin{multline*}
\bm{t}^{-}\bm{G}_{\theta}^{\ast}\bm{s}
=
\bigoplus_{k=0}^{\min(m,n)}
(\bm{r}^{-}\bm{C}_{1}
\oplus
\bm{h}^{-})
(\bm{D}_{1}^{-}\bm{C}_{\theta})^{k}
\bm{g}
\\\oplus
\bigoplus_{k=0}^{\min(m,n)}
(\bm{h}^{-}\bm{D}_{1}^{-}
\oplus
\bm{r}^{-})
(\bm{C}_{\theta}\bm{D}_{1}^{-})^{k}
\bm{q}
\oplus
\bigoplus_{k=0}^{\min(m,n)}
\theta^{-1}
\bm{r}^{-}
\bm{A}(\bm{D}_{1}^{-}\bm{C}_{\theta})^{k}
\bm{g}.
\end{multline*}

By using the identities $\bm{D}_{1}^{-}\bm{C}_{\theta}=\theta^{-1}\bm{R}\oplus\bm{S}$ and $\bm{C}_{\theta}\bm{D}_{1}^{-}=\theta^{-1}\bm{P}\oplus\bm{Q}$, we represent the inequality as
\begin{multline*}
\bigoplus_{k=0}^{\min(m,n)}
(\bm{r}^{-}\bm{C}_{1}
\oplus
\bm{h}^{-})
(\theta^{-1}\bm{R}\oplus\bm{S})^{k}
\bm{g}
\\\oplus
\bigoplus_{k=0}^{\min(m,n)}
(\bm{h}^{-}\bm{D}_{1}^{-}
\oplus
\bm{r}^{-})
(\theta^{-1}\bm{P}\oplus\bm{Q})^{k}
\bm{q}
\oplus
\bigoplus_{k=0}^{\min(m,n)}
\theta^{-1}
\bm{r}^{-}\bm{A}
(\theta^{-1}\bm{R}\oplus\bm{S})^{k}
\bm{g}
\leq
\mathbb{1},
\end{multline*}
which is equivalent to the system
\begin{equation}
\begin{aligned}
(\bm{r}^{-}\bm{C}_{1}
\oplus
\bm{h}^{-})\bm{g}
\oplus
(\bm{h}^{-}\bm{D}_{1}^{-}
\oplus
\bm{r}^{-})\bm{q}
&\leq
\mathbb{1},
\\
\bigoplus_{k=1}^{\min(m,n)}
(\bm{r}^{-}\bm{C}_{1}
\oplus
\bm{h}^{-})
(\theta^{-1}\bm{R}\oplus\bm{S})^{k}
\bm{g}
&\leq
\mathbb{1},
\\
\bigoplus_{k=1}^{\min(m,n)}
(\bm{h}^{-}\bm{D}_{1}^{-}
\oplus
\bm{r}^{-})
(\theta^{-1}\bm{P}\oplus\bm{Q})^{k}
\bm{q}
&\leq
\mathbb{1},
\\
\bigoplus_{k=0}^{\min(m,n)}
\theta^{-1}
\bm{r}^{-}\bm{A}
(\theta^{-1}\bm{R}\oplus\bm{S})^{k}
\bm{g}
&\leq
\mathbb{1}.
\end{aligned}
\label{I-tGs1}
\end{equation}

From the assumption of the theorem and the definition of the Kleene star matrix which yields $\bm{Q}^{\ast}\geq\bm{I}$, $\bm{S}^{\ast}\geq\bm{I}$, it follows that the first inequality in the system is satisfied. We examine the second inequality and apply identity \eqref{E-oplus1pABk} to rewrite this inequality as
\begin{multline*}
\bigoplus_{k=1}^{\min(m,n)}
\bigoplus_{0\leq i_{0}+i_{1}+\cdots+i_{k}\leq \min(m,n)-k}
\theta^{-k}
(\bm{r}^{-}\bm{C}_{1}
\oplus
\bm{h}^{-})
\bm{S}^{i_{0}}(\bm{R}\bm{S}^{i_{1}}\cdots\bm{R}\bm{S}^{i_{k}})
\bm{g}
\\\oplus
\bigoplus_{k=1}^{\min(m,n)}
(\bm{r}^{-}\bm{C}_{1}
\oplus
\bm{h}^{-})
\bm{S}^{k}
\bm{g}
\leq
\mathbb{1},
\end{multline*}
which is equivalent to the system of inequalities
\begin{align*}
\bigoplus_{k=1}^{\min(m,n)}
\bigoplus_{0\leq i_{0}+i_{1}+\cdots+i_{k}\leq \min(m,n)-k}
\theta^{-k}
(\bm{r}^{-}\bm{C}_{1}
\oplus
\bm{h}^{-})
\bm{S}^{i_{0}}(\bm{R}\bm{S}^{i_{1}}\cdots\bm{R}\bm{S}^{i_{k}})
\bm{g}
&\leq
\mathbb{1},
\\
\bigoplus_{k=1}^{\min(m,n)}
(\bm{r}^{-}\bm{C}_{1}
\oplus
\bm{h}^{-})
\bm{S}^{k}
\bm{g}
&\leq
\mathbb{1}.
\end{align*}

The second inequality in the last system is valid by the assumption of the theorem. The solution of the first with respect to $\theta$ is given by
\begin{equation*}
\theta
\geq
\bigoplus_{k=1}^{\min(m,n)}
\bigoplus_{0\leq i_{0}+i_{1}+\cdots+i_{k}\leq \min(m,n)-k}
((\bm{r}^{-}\bm{C}_{1}
\oplus
\bm{h}^{-})
\bm{S}^{i_{0}}(\bm{R}\bm{S}^{i_{1}}\cdots\bm{R}\bm{S}^{i_{k}})
\bm{g})^{1/k}.
\end{equation*}

In the same way, we solve the third and fourth inequalities at \eqref{I-tGs1} to obtain
\begin{align*}
\theta
&\geq
\bigoplus_{k=1}^{\min(m,n)}
\bigoplus_{0\leq i_{0}+i_{1}+\cdots+i_{k}\leq \min(m,n)-k}
((\bm{h}^{-}\bm{D}_{1}^{-}
\oplus
\bm{r}^{-})
\bm{Q}^{i_{0}}(\bm{P}\bm{Q}^{i_{1}}\cdots\bm{P}\bm{Q}^{i_{k}})
\bm{q})^{1/k},
\\
\theta
&\geq
\bigoplus_{k=0}^{\min(m,n)}
\bigoplus_{0\leq i_{0}+i_{1}+\cdots+i_{k}\leq \min(m,n)-k}
(\bm{r}^{-}\bm{A}
\bm{S}^{i_{0}}(\bm{R}\bm{S}^{i_{1}}\cdots\bm{R}\bm{S}^{i_{k}})
\bm{g})^{1/(k+1)}.
\end{align*}

By combining all lower bounds derived for $\theta$, we obtain the minimum $\eta$ in the problem given by \eqref{e:eta}.

Let us now consider the solution given by \eqref{E-z}, where we replace $\theta$ by $\eta$. Specifically, with $\bm{C}_{\eta}=\eta^{-1}\bm{A}\oplus\bm{C}_{1}$, we write
\begin{equation*}
\bm{G}_{\eta}^{\ast}
=
\begin{pmatrix}
(\bm{D}_{1}^{-}\bm{C}_{\eta})^{\ast} & (\bm{D}_{1}^{-}\bm{C}_{\eta})^{\ast}\bm{D}_{1}^{-}
\\
(\bm{C}_{\eta}\bm{D}_{1}^{-})^{\ast}\bm{C}_{\eta} & (\bm{C}_{\eta}\bm{D}_{1}^{-})^{\ast}
\end{pmatrix}
=
\begin{pmatrix}
(\bm{D}_{1}^{-}\bm{C}_{\eta})^{\ast} & \bm{D}_{1}^{-}(\bm{C}_{\eta}\bm{D}_{1}^{-})^{\ast}
\\
\bm{C}_{\eta}(\bm{D}_{1}^{-}\bm{C}_{\eta})^{\ast} & (\bm{C}_{\eta}\bm{D}_{1}^{-})^{\ast}
\end{pmatrix}.
\end{equation*}

With $\bm{D}_{1}^{-}\bm{C}_{\eta}=\eta^{-1}\bm{R}\oplus\bm{S}$ and $\bm{C}_{\eta}\bm{D}_{1}^{-}=\eta^{-1}\bm{P}\oplus\bm{Q}$, we finally have
\begin{align*}
\bm{x}
&=
(\eta^{-1}\bm{R}\oplus\bm{S})^{\ast}
(\bm{u}
\oplus
\bm{D}_{1}^{-}\bm{v}),
\\
\bm{y}
&=
(\eta^{-1}\bm{P}\oplus\bm{Q})^{\ast}
((\eta^{-1}\bm{A}\oplus\bm{C}_{1})\bm{u}
\oplus
\bm{v}).
\end{align*}

The conditions on the vectors of parameters $\bm{u}$ and $\bm{v}$ take the form
\begin{align*}
\bm{g}
&\leq
\bm{u}
\leq
((\bm{h}^{-}
\oplus
\bm{r}^{-}
(\eta^{-1}\bm{A}\oplus\bm{C}_{1}))(\eta^{-1}\bm{R}\oplus\bm{S})^{\ast})^{-},
\\
\bm{q}
&\leq
\bm{v}
\leq
((\bm{h}^{-}
\bm{D}_{1}^{-}
\oplus
\bm{r}^{-})
(\eta^{-1}\bm{P}\oplus\bm{Q})^{\ast})^{-},
\end{align*}
which completes the proof.
\qed
\end{proof}

\section{Overall solution and computational complexity}
\label{s:complexity}

In this section, we summarize the solution offered by Lemma~\ref{L:lower-level} and Theorem~\ref{T:upper-level} as a complete computational scheme and evaluate its computational complexity.

\begin{enumerate}
\item
The solution of the first-stage problem involves the following phases.
\begin{enumerate}
\item
The solution starts with checking the existence condition
\begin{equation*}
\bm{h}^{-}\bm{g}\oplus(\bm{h}^{-}\bm{D}^{-}\oplus\bm{r}^{-})\bm{q}
\leq
\mathbb{1}.
\end{equation*}
\item
If this condition is met, we compute the minimum value of the problem
\begin{multline*}
\mu
=
\rho(\bm{C}\bm{D}^{-})
\oplus
\bigoplus_{k=1}^{\min(m,n)}
\left(
\bm{h}^{-}(\bm{D}^{-}\bm{C})^{k}\bm{g}
\right)^{1/k}
\\\oplus
\bigoplus_{k=1}^{\min(m,n)}
\left(
(\bm{h}^{-}\bm{D}^{-}
\oplus
\bm{r}^{-})(\bm{C}\bm{D}^{-})^{k}\bm{q}
\right)^{1/k}
\oplus
\bigoplus_{k=0}^{\min(m,n)}
(\bm{r}^{-}\bm{C}(\bm{D}^{-}\bm{C})^{k}\bm{g})^{1/(k+1)},
\end{multline*}
otherwise the first-stage problem does not have regular solutions.
\end{enumerate}
\item
The solution of the second-stage problem goes as follows.
\begin{enumerate}
\item
The solution starts with calculating the matrices
\begin{gather*}
\bm{D}_{1}^{-}
=
\bm{B}^{-}\oplus\bm{D}^{-},
\qquad
\bm{C}_{1}
=
\mu^{-1}\bm{C},
\\
\bm{P}
=
\bm{A}\bm{D}_{1}^{-},
\qquad
\bm{Q}
=
\bm{C}_{1}\bm{D}_{1}^{-},
\qquad
\bm{R}
=
\bm{D}_{1}^{-}\bm{A},
\qquad
\bm{S}
=
\bm{D}_{1}^{-}\bm{C}_{1}.
\end{gather*}
\item
Next, we check the existence condition
\begin{equation*}
\mathop\mathrm{Tr}(\bm{Q})
\oplus
(\bm{h}^{-}\bm{D}_{1}^{-}
\oplus
\bm{r}^{-})\bm{Q}^{\ast}\bm{q}
\oplus
(\bm{r}^{-}\bm{C}_{1}
\oplus
\bm{h}^{-})\bm{S}^{\ast}\bm{g}
\leq
\mathbb{1}.
\end{equation*}
\item
If this condition is met, we compute the minimum value of the problem
\begin{multline*}
\eta
=
\bigoplus_{k=1}^{\min(m,n)}
\bigoplus_{0\leq i_{0}+i_{1}+\cdots+i_{k}\leq\min(m,n)-k}
\mathop\mathrm{tr}\nolimits^{1/k}
(\bm{Q}^{i_{0}}(\bm{P}\bm{Q}^{i_{1}}\cdots\bm{P}\bm{Q}^{i_{k}}))
\\\oplus
\bigoplus_{k=1}^{\min(m,n)}
\bigoplus_{0\leq i_{0}+i_{1}+\cdots+i_{k}\leq \min(m,n)-k}
((\bm{r}^{-}\bm{C}_{1}
\oplus
\bm{h}^{-})
\bm{S}^{i_{0}}(\bm{R}\bm{S}^{i_{1}}\cdots\bm{R}\bm{S}^{i_{k}})
\bm{g})^{1/k}
\\\oplus
\bigoplus_{k=1}^{\min(m,n)}
\bigoplus_{0\leq i_{0}+i_{1}+\cdots+i_{k}\leq \min(m,n)-k}
((\bm{h}^{-}\bm{D}_{1}^{-}
\oplus
\bm{r}^{-})
\bm{Q}^{i_{0}}(\bm{P}\bm{Q}^{i_{1}}\cdots\bm{P}\bm{Q}^{i_{k}})
\bm{q})^{1/k}
\\\oplus
\bigoplus_{k=0}^{\min(m,n)}
\bigoplus_{0\leq i_{0}+i_{1}+\cdots+i_{k}\leq \min(m,n)-k}
(\bm{r}^{-}\bm{A}
\bm{S}^{i_{0}}(\bm{R}\bm{S}^{i_{1}}\cdots\bm{R}\bm{S}^{i_{k}})
\bm{g})^{1/(k+1)},
\end{multline*}
otherwise the two-stage problem does not have regular solutions.
\item
To obtain a solution, we first evaluate upper bounds in the conditions
\begin{align*}
\bm{g}
&\leq
\bm{u}
\leq
((\bm{h}^{-}
\oplus
\bm{r}^{-}
(\eta^{-1}\bm{A}\oplus\bm{C}_{1}))(\eta^{-1}\bm{R}\oplus\bm{S})^{\ast})^{-},
\nonumber\\
\bm{q}
&\leq
\bm{v}
\leq
((\bm{h}^{-}
\bm{D}_{1}^{-}
\oplus
\bm{r}^{-})
(\eta^{-1}\bm{P}\oplus\bm{Q})^{\ast})^{-}.
\end{align*}
\item
Taking vectors $\bm{u}$ and $\bm{v}$ to satisfy these conditions, we find a solution of the entire problems by writing
\begin{align*}
\bm{x}
&=
(\eta^{-1}\bm{R}\oplus\bm{S})^{\ast}
(\bm{u}
\oplus
\bm{D}_{1}^{-}\bm{v}),
\nonumber\\
\bm{y}
&=
(\eta^{-1}\bm{P}\oplus\bm{Q})^{\ast}
((\eta^{-1}\bm{A}\oplus\bm{C}_{1})\bm{u}
\oplus
\bm{v}).
\end{align*}
\end{enumerate}
\end{enumerate}

Below, we evaluate the computational complexity of the solution by estimating the complexity which is involved in the solution of the first- and second-stage problems.

\subsection{First-stage problem}

We start with estimating the complexity of the calculation of the minimum $\mu$. We consider the first term in the form of the spectral radius $\rho(\bm{C}\bm{D}^{-})=\rho(\bm{D}^{-}\bm{C})$ and observe that calculating the matrices $\bm{C}\bm{D}^{-}$ and $\bm{D}^{-}\bm{C}$ respectively requires $O(m^{2}n)$ and $O(mn^{2})$ operations.

To reduce calculation, we use the matrix of the minimal order equal to $\min(m,n)$, which involves less number of operations given by $O(mn\min(m,n))$. Since the evaluation of spectral radius for this matrix using Karp's algorithm requires $O(\min(m,n)^{3})$ operations, which is not greater than $O(mn\min(m,n))$, the computational complexity of the first term is $O(mn\min(m,n))$.

Let us verify that the complexity of  calculating the other two terms has the same order. To examine the complexity for these terms, it is sufficient to consider only one of them, say the sum
\begin{equation*}
\bigoplus_{k=1}^{\min(m,n)}
\left(
\bm{h}^{-}(\bm{D}^{-}\bm{C})^{k}\bm{g}
\right)^{1/k},
\end{equation*}
since the other sum requires the same order of operations. Consider the term under summation and suppose that $m\leq n$. In this case, the calculation of the term $\bm{h}^{-}(\bm{D}^{-}\bm{C})^{k}\bm{g}$ for $k>1$ involves the following series of left multiplication of $(m\times m)$-matrices by $m$-vectors:
\begin{equation*}
(\bm{h}^{-}\bm{D}^{-})(\bm{C}\bm{D}^{-}),(\bm{h}^{-}\bm{D}^{-}\bm{C}\bm{D}^{-})(\bm{C}\bm{D}^{-}),\ldots,
\end{equation*}
each taking $O(m^{2})$ operations. As a result, the entire sum requires $O(m^{3})$ operations.

Under the condition $m>n$, the calculation of this term can be arranged as the series of multiplications
\begin{equation*}
\bm{h}^{-}(\bm{D}^{-}\bm{C}),(\bm{h}^{-}\bm{D}^{-}\bm{C})(\bm{D}^{-}\bm{C}),\ldots,
\end{equation*}
which leads to the complexity of order $O(n^{3})$. We combine both orders as $O(\min(m,n)^{3})$.

The number $O(mn\min(m,n))$ of operations required to calculate the smaller of the matrices $\bm{C}\bm{D}^{-}$ and $\bm{D}^{-}\bm{C}$ is not less than $O(\min(m,n)^{3})$. As a result, the order of complexity for the sum considered is $O(mn\min(m,n))$, which also represents the overall order of complexity for $\mu$.

\subsection{Second-stage problem}

We now evaluate the computational complexity of the solution of the second-stage problem. We observe that calculating the matrices $\bm{P}=\bm{A}\bm{D}_{1}^{-}$ and $\bm{Q}=\bm{C}_{1}\bm{D}_{1}^{-}$ of order $m$ involves $O(m^{2}n)$ operations, and the matrices $\bm{R}=\bm{D}_{1}^{-}\bm{A}$ and $\bm{S}=\bm{D}_{1}^{-}\bm{C}_{1}$ of order $n$ involves $O(n^{2}m)$ operations.

First, we consider the existence condition and observe that
\begin{equation*}
\mathop\mathrm{Tr}(\bm{Q})
=
\mathop\mathrm{Tr}(\bm{C}_{1}\bm{D}_{1}^{-})
=
\mathop\mathrm{Tr}(\bm{D}_{1}^{-}\bm{C}_{1})
=
\mathop\mathrm{Tr}(\bm{S}).
\end{equation*}

Similarly to the first stage, we choose out of $\bm{Q}$ and $\bm{S}$ the matrix which has the smaller order $\min(m,n)$, and observe that the calculation of this matrix requires $O(mn\min(m,n))$ operations. Let us verify that the calculation of the left-hand side of the existence condition involves no more than $O(mn)+O(\min(m,n)^3)$ operations.

The calculation of Kleene matrix using the Floyd-Warshall algorithm has the computational complexity of the third degree. Therefore, if $m\leq n$, we can obtain $\bm{Q}^{\ast}$ in $O(m^{3})$ operations. Since $\mathop\mathrm{Tr}(\bm{Q})=\mathop\mathrm{tr}(\bm{Q}\bm{Q}^{\ast})$, calculating $\mathop\mathrm{Tr}(\bm{Q})$ requires $O(m^{3})$ operations as well. Given a matrix $\bm{Q}^{\ast}$, the calculation of the scalar $(\bm{h}^{-}\bm{D}_{1}^{-}\oplus\bm{r}^{-})\bm{Q}^{\ast}\bm{q}$ involves $O(mn)$ operations.

We now consider the matrix $\bm{S}^{\ast}$ and write
\begin{equation*}
\bm{S}^{\ast}
=
\bigoplus_{k=0}^{n-1}
\bm{S}^{k}
=
\bigoplus_{k=0}^{n}
\bm{S}^{k}
=
\bm{I}
\oplus
\bm{D}_{1}^{-}
\bigoplus_{k=0}^{n-1}
(\bm{C}_{1}
\bm{D}_{1}^{-})^{k}
\bm{C}_{1}
=
\bm{I}
\oplus
\bm{D}_{1}^{-}
\bm{Q}^{\ast}
\bm{C}_{1}.
\end{equation*}
Using this expression we write
\begin{equation*}
(\bm{r}^{-}\bm{C}_{1}
\oplus
\bm{h}^{-})\bm{S}^{\ast}\bm{g}=(\bm{r}^{-}\bm{C}_{1}
\oplus
\bm{h}^{-})(\bm{I}
\oplus
\bm{D}_{1}^{-}
\bm{Q}^{\ast}
\bm{C}_{1})\bm{g}\end{equation*}

The calculation of this scalar also takes $O(mn)$ operations (being dominated by the calculation of $\bm{h}^-\bm{D}_1^-$ and
$\bm{C}_{1}\bm{g}$). Thus the existence condition can be verified in $O(mn)+O(m^3)$ operations if $m\leq n$. Combining this with a similar estimate for the case $n\leq m$, we obtain that verifying the existence condition takes $O(mn)+O(\min(m,n)^3)$ operations, provided that $\bm{Q}$ or $\bm{S}$ have been computed previously.


To evaluate the complexity of calculating the minimum $\eta$, we examine four matrix sums, which add up to provide $\eta$. We estimate the number of operations required to calculate each of the sums by applying Corollary~\ref{C:complexity} with $p$ replaced by $\min(m,n)$ and $n$ by $\max(m,n)$.

We observe that taking traces of matrices, left and right multiplication of matrices by vectors and taking roots of the scalars obtained do not increase the order of computational complexity which is determined by the number of summands in the matrix sums. After application of Corollary~\ref{C:complexity}, we find the computational complexity of calculating $\eta$ does not exceed $O(\min(m,n)^{2}\max(m,n)^{3})$ or equivalently $O(m^{2}n^{2}\max(m,n))$. This order is higher than those involved in the other parts of solution, and thus gives an upper bound for the computational complexity of the overall solution obtained for the two-stage problem under consideration.

The remarks above refer to finding the optimal value of the problem. As for the set of optimal solutions, we note that
it is a closed and bounded tropically convex set,  which is also convex in the usual sense and thus an alcoved polyhedron \cite{MariaJesus}. In particular, the description that we give in \eqref{e:xy} and \eqref{e:uv} enables one to write out extremal points of that tropical convex set, see e.\,g., \cite{GK-07} and \cite{KNS-14}, for definitions and more background on tropical convexity and tropical extremal points. The extremal points of solution set described by \eqref{e:xy} and \eqref{e:uv} are determined by the upper and lower bounds of the box described by \eqref{e:uv} and their number does not exceed $m+n+1$.

\section{Discussion}

Theorem \ref{T:upper-level} provides an explicit solution to \eqref{P-Bi-Level}. The same technique can be used to solve some extended versions of \eqref{P-Bi-Level}: note that we can add constraints in the form $\bm{E}\bm{x}\leq\bm{y}$ or $\bm{E}\bm{u}\leq\bm{v}$, where $\bm{E}$ is a matrix of appropriate dimensions, or consider a number of similar projects on the first stage instead of just one project. Since the techniques that would be applied and the results of such application are very similar to what is presented above, we are not giving further details of such extensions here.

Unlike the problem that we considered, the two-stage optimization problems considered in the literature usually have a parametric optimization problem on the first stage, where some of the parameters are second-stage decision variables. The problems of this kind, with the objective functions as in the present paper and possibly more elaborate constraints, offer a promising direction for research.

Furthemore, the problem solved in the present paper could be rather easily reduced to a one-stage optimization problem, since the solution set of the first-stage problem could be concisely and explicitly written out.  In the two-stage optimization, the reduction of two-stage problem to a one-stage problem (if it exists) is often due to duality and complementary slackness conditions, which are lacking in tropical mathematics. Also, such reduction often introduces non-linear constraints, which was the case for the tropical bi-level optimization problems considered in~\cite{SLiu-20}, and this will very likely present a formidable challenge for the future development of two-stage and multi-stage optimization problems in tropical mathematics.

\section*{Acknowledgement}
The work of the second author was supported by EPSRC Grant EP/P019676/1. The authors would like to thank Dr. \v{S}tefan Bere\v{z}n\'{y}, with whom some aspects of this work were discussed during his stay in Birmingham in 2018.

\bibliographystyle{abbrvurl}
\bibliography{Minimizing_maximum_lateness_in_bi-level_projects_by_tropical_optimization}

\end{document}